\documentclass[12pt,a4paper,english,reqno]{amsart}
\usepackage[a4paper,footskip=1.5em]{geometry}
\usepackage{amsmath,amssymb,amsthm,mathtools,comment}
\usepackage[mathscr]{euscript}
\usepackage{tikz,babel,enumitem,adjustbox,}
\usepackage[final]{microtype}
\usepackage[numbers]{natbib}
\usetikzlibrary{decorations.pathreplacing,calc,arrows}
\usepackage{hyperref}
\hypersetup{colorlinks=true,linkcolor=blue,citecolor=blue,pdfpagemode=UseNone,pdfstartview={XYZ null null 1.00}}
\usepackage{cmtiup}

\allowdisplaybreaks

\theoremstyle{plain}
\newtheorem{theorem}{Theorem}[section]

\newtheorem{conjecture}[theorem]{Conjecture}

\theoremstyle{remark}

\newcommand{\eps}{\ensuremath{\varepsilon}}

\def\N{\mathbb{N}}

\let\emptyset\varnothing

\let\originalleft\left
\let\originalright\right
\renewcommand{\left}{\mathopen{}\mathclose\bgroup\originalleft}
\renewcommand{\right}{\aftergroup\egroup\originalright}

\makeatletter
\def\imod#1{\allowbreak\mkern10mu({\operator@font mod}\,\,#1)}
\makeatother

\begin{document}
\title{Disproportionate division}

\author{Logan Crew}
\address{Department of Mathematics, University of Pennsylvania, Philadelphia, PA 19104, USA}
\email{crewl@math.upenn.edu}

\author{Bhargav Narayanan}
\address{Department of Mathematics, Rutgers University, Piscataway, NJ 08854, USA}
\email{narayanan@math.rutgers.edu}

\author{Sophie Sprikl}
\address{Department of Mathematics, Rutgers University, Piscataway, NJ 08854, USA}
\email{sophie.spirkl@rutgers.edu}

\date{5 September 2019}

\begin{abstract}
	We study the disproportionate version of the classical cake-cutting problem: how efficiently can we divide a cake, here $[0,1]$, among $n$ agents with different demands $\alpha_1, \alpha_2, \dots, \alpha_n$ summing to $1$? When all the agents have equal demands of $\alpha_1 = \alpha_2 = \dots = \alpha_n = 1/n$, it is well-known that there exists a fair division with $n-1$ cuts, and this is optimal. For arbitrary demands on the other hand, folklore arguments from algebraic topology show that $O(n\log n)$ cuts suffice, and this has been the state of the art for decades. Here, we improve the state of affairs in two ways: we prove that disproportionate division may always be achieved with $3n-4$ cuts, and give an effective combinatorial procedure to construct such a division. We also offer a topological conjecture that implies that $2n-2$ cuts suffice in general, which would be optimal.
\end{abstract}

\maketitle

\section{Introduction}
The cake-cutting problem, whose study was initiated by Banach and Steinhaus~\citep{steinhaus} in 1949, is a classical measure partitioning problem concerned with the division of a `cake', here the unit interval $[0,1]$, amongst $n \ge 2$ agents each with their own `utilities', here non-atomic Borel probability measures $\mu_1, \mu_2, \dots, \mu_n$ on $[0,1]$.

Traditionally, one is interested in \emph{fair divisions}: a partition $X_1 \cup X_2 \dots \cup X_n$ of $[0,1]$ is said to be a \emph{fair division} if $\mu_i (X_i) \ge 1/n$ for all $1 \le i \le n$. That a fair division always exists is a classical fact, and is more or less trivial to see. Slide a knife from $0$ to $1$, stopping at the first point $x \in [0,1]$ where $\mu_i ([0,x)) \ge 1/n$ for some $1 \le i \le n$, set $X_i = [0,x)$, and recurse on $[x,1]$ with the remaining $n-1$ agents; this in fact shows us that there is a fair division with $n-1$ cuts, which is optimal since $n-1$ cuts are necessary just to partition the unit interval into $n$ pieces. There is now a significant body of work on fair divisions, investigating various aspects of this problem such as envy-free-ness, equitability and computational complexity, and utilising a wide variety of tools from combinatorics, topology and game theory; for a small sample of the literature, see~\citep{dubspan, strom, webb, better, envy1, envy2}.

Our focus here is a natural `disproportionate' generalisation of the fair division problem: what can we say when the $n$ agents have differing claims to the cake? In the measure partitioning literature, the disproportionate problem often presents subtleties not inherent in the proportionate problem (notably in~\citep{stonetukey, imre}, for example), and as we shall shortly see, this is also the case here.

Given non-negative \emph{demands} $\alpha_1, \alpha_2, \dots, \alpha_n$ summing to $1$, a \emph{disproportionate division} for these demands is a partition $X_1 \cup X_2 \dots \cup X_n$ of $[0,1]$ with $\mu_i (X_i) \ge \alpha_i$ for all $1 \le i \le n$. The requirement in this definition that all the demands sum to $1$ is sensible since the measures $\mu_1, \mu_2, \dots, \mu_n$ are, in general, allowed to be identical.

Disproportionate divisions are easily seen to exist when all the demands are rational: bring all the demands to a common denominator, say $D$, construct a fair division for $D$ agents, and then distribute the resulting pieces appropriately. However, this does not immediately demonstrate existence for irrational demands, since taking limits with respect to the Hausdorff metric, for example, does not preserve measure. Nonetheless, it can be shown using an infinite analogue of the sliding knife procedure that disproportionate divisions exist for all demands, though when proceeding thusly, one cannot in general avoid `crumbs': the pieces of cake produced by this argument are only guaranteed to be countable unions of intervals.

What if we ask for an \emph{efficient} disproportionate division with a \emph{bounded} number of cuts? A beautiful topological result of Stromquist and Woodall~\citep{sw} furnishes an answer; a folklore argument utilising this result shows that a disproportionate division for $n$ agents with arbitrary demands may always be found with $O(n^2)$ cuts, and a more efficient rendition of this argument, recently discovered by Segal-Halevi~\citep{halevi}, shows that in fact $O(n\log n)$ cuts always suffice. Our main result improves on these decades-old topological arguments as follows.

\begin{theorem}\label{mainthm}
	For all $n\ge 2$, given non-atomic probability measures $\mu_1, \mu_2, \dots, \mu_n$ on $[0,1]$ and non-negative reals $\alpha_1, \alpha_2, \dots, \alpha_n$ summing to $1$, there exists a disproportionate division for these demands with at most $3n-4$ cuts.
\end{theorem}

Again, it is clear that $n-1$ cuts are always necessary for $n$ agents, so this result is tight up to multiplicative constants. Our proof of Theorem~\ref{mainthm} is combinatorial as opposed to topological; an attractive byproduct of this approach is that the proof is constructive, yielding an effective procedure for disproportionate division.

While Theorem~\ref{mainthm} determines the optimal number of cuts for disproportionate division with $n$ agents up to multiplicative constants, the problem of pinning down this extremal number still remains. Unlike with fair division, it turns out that $n-1$ cuts do not always suffice; there is a construction demonstrating that $2n-2$ cuts may be necessary in general. We suspect that this construction, not Theorem~\ref{mainthm}, reflects the truth, and that the tightness of this construction should follow from topological considerations: to this end, we shall present a topological conjecture that implies that $2n-2$ cuts always suffice for disproportionate division with $n$ agents.

This paper is organised as follows. We give the proof of Theorem~\ref{mainthm} in Section~\ref{s:proof}. We conclude by discussing lower bounds, and by presenting a measure partitioning conjecture implying optimal bounds, in Section~\ref{s:conc}.

\section{Proof of the main result}\label{s:proof}
In this section, we prove our main result. It will be helpful to have some notation. As is usual, we write $[n]$ for the set $\{1, 2, \dots, n\}$. Let us also write $f(n)$ for the maximum number of cuts needed (over all possible measures and demands) for disproportionate division with $n$ agents; as noted earlier, $f(n)$ exists for all $n \in \N$ and we have the estimate $f(n) = O(n \log n)$. Recall that our result, in this language, asserts that $f(n) \le 3n - 4$ for all $n \ge 2$.

Before we turn to the proof, let us briefly highlight the main idea behind the proof. We shall proceed by induction on $n$. A natural first attempt is as follows. As in the sliding knife procedure, make $O(1)$ cuts to
\begin{itemize}
	\item partition the unit interval $[0,1]$ into two pieces $A$ and $B$, and
	\item find a suitable partition of the $n$ agents into two groups of size $a$ and $b$ with $a + b = n$,
\end{itemize}
so that we may recursively solve the division problem for the first group of size $a$ on $A$ and the second group of size $b$ on $B$. This would yield the estimate $f(n) \le f(a) + f(b) + O(1)$, which would in turn imply that $f(n) = O(n)$. While almost all the arguments in the literature (that we are aware of) follow this rough outline, we have however been unable to find a proof along these lines. Instead, it turns out to be quite helpful to allow ourselves a little more elbow room: once we have $A$ and $B$, instead of attempting to partition the $n$ agents into two groups (thereby generating two subproblems whose sizes add up to $n$), we shall allow some agents to participate in both subproblems (i.e., on both $A$ and $B$), thereby producing subproblems {whose sizes add up to more than $n$}; we shall then make up for this inefficiency by ensuring that these subproblems are both suitably smaller than $n$.

\begin{proof}[Proof of Theorem~\ref{mainthm}]
	First, a matter of convenience: since all our measures are non-atomic, open, half-open and closed intervals with the same endpoints all have the same measure in each of our measures, so we shall use these interchangeably as appropriate.

	Let us record two simple facts to start with. We trivially have $f(1) = 0$. Next, note that $f(2) \le 2$, which we may infer from the following fact: given two non-atomic probability measures $\mu$ and $\mu'$ on the unit circle $S^1$ and any $\alpha \in [0,1]$, there is an interval $X \subset S^1$ such that $\mu(X)  = \alpha$ and $\mu'(X) \ge \alpha$. To see this, note that, by compactness, it suffices to prove the claim for rational $\alpha$; indeed, the limit of a sequence of closed intervals with respect to the Hausdorff metric is a closed interval. Writing $\alpha = p/q$, we may find intervals $X_1, X_2, \dots, X_q$ so that $\mu(X_i) = \alpha$ for each $ 1\le i \le q$, and so each point of the circle is covered by exactly $p$ of these intervals; by pigeonholing, it is now clear that $\mu'(X_i) \ge \alpha$ for some $1\le i \le q$.

	With the above bounds for $f(1)$ and $f(2)$ in hand, the result easily follows, by induction, from the following estimate: for each $n \ge 2$,
	\begin{equation}\label{rec} f(n+1) \le \max\{1+f(n), \max_{2 \le k \le n} \{1+f(k) + f(n+2-k)\}\}. \tag{$\dagger$}\end{equation}

	We prove the above estimate as follows. Given measures $\mu_1, \mu_2, \dots, \mu_{n+1}$ and demands $\alpha_1, \alpha_2, \dots, \alpha_{n+1}$ with $n \ge 2$, define for $\vartheta \in [0,1]$ and $k \in [n+1]$, the set
	\[S(\vartheta,k) = \{i : \mu_i ([0,\vartheta]) > 1/2\} \cup \{i : \mu_i ([0,\vartheta]) = 1/2 \text{ and } i \le k\}.\]
	Now, fix a minimal $x \in [0,1]$ and then a minimal $t \in [n+1]$ so that
	\[ \sum_{i \in S(x,t)} \alpha_i \ge 1/2. \]

	Next, set $P = S(x,t) \setminus \{t\}$ and $Q = [n+1] \setminus S(x,t)$, whence $[n+1] = P \cup Q \cup \{t\} $. By the minimality of $x$ and $t$, note that $\mu_i([0,x]) \ge 1/2$ for all $i \in P$, $\mu_i([0,x]) \le 1/2$ for all $i \in Q$, and $\mu_t([0,x]) = 1/2$; furthermore, we also have $\sum_{i \in P} \alpha_i \le 1/2$ and $\sum_{i \in Q} \alpha_i \le 1/2$.

	If both $P$ and $Q$ are non-empty, then we are done. Indeed, set $\alpha' = 1/2 - \sum_{i \in P} \alpha_i$ and $\alpha'' = 1/2 - \sum_{i \in Q} \alpha_i$ so that $\alpha', \alpha'' \ge 0$ and $\alpha' + \alpha'' = \alpha_t$. Now, note that the division problem on $[0,x]$ for the measures in $P \cup \{t\}$, with the original demands $\alpha_i$ for each measure $\mu_i$ with $i \in P$, and demand $\alpha'$ for the measure $\mu_t$, is (after rescaling by constants) a smaller instance of the disproportionate division problem for $k = |P| + 1$ measures, with $2 \le k \le n$. Similarly, the division problem on $(x,1]$ for the measures in $Q \cup \{t\}$, with the original demands $\alpha_i$ for each measure $\mu_i$ with $i \in Q$, and demand $\alpha''$ for the measure $\mu_t$, is again a smaller instance of the disproportionate division problem for $n+2-k$ measures. By making a cut at $x$ and solving these two smaller division problems, we see that~\eqref{rec} holds in this case.

	Now, assume that one of $P$ or $Q$ is empty. We claim that it must have been that $\alpha_t \ge 1/2$. Indeed, if $P = \emptyset$, then $S(x,t) = \{t\}$ and the claim follows from the definition of $S(x,t)$, and if $Q = \emptyset$, then $\alpha_t = 1 - \sum_{i \in P} \alpha_i \ge 1-1/2 = 1/2$.

	If $\alpha_t = 1/2$, then we are again done. Indeed, if $P = \emptyset$, then assign the piece $[0,x]$ to $\mu_t$; having made one cut at $x$, we are now left with the division problem on $(x,1]$ for the $n$ measures $\mu_i$ with $i \in [n+1] \setminus \{t\}$, establishing~\eqref{rec}. If $Q = \emptyset$, then we proceed similarly by assigning the piece $(x,1]$ to $\mu_t$.

	We are left to address the case where $\alpha_t > 1/2$ and one of $P$ or $Q$ is empty. We now describe how to establish~\eqref{rec} when $P = \emptyset$; the other case may be handled analogously by exchanging the roles of $P$ and $Q$.

	Since $P = \emptyset$, note that $\mu_t([0,x]) = 1/2$ and $\mu_i ([0,x]) \le 1/2$ for each $i \ne t$. We know that $\alpha_t > 1/2$, so we may fix a minimal $y \in (x,1]$ such that $\mu_t([0,y]) = \alpha_t$, and define $U = \{i : \mu_i ([0,y]) \ge \alpha_t \text{ and } i \ne t\}$ and $V = \{i : \mu_i ([0,y]) < \alpha_t \text{ and } i \ne t\}$, so that $U$ and $V$ partition $Q = [n+1] \setminus \{t\}$.

	If $U = \emptyset$, we are done as before by cutting at $y$, assigning $[0,y]$ to $\mu_t$ and recursively solving the division problem on $(y,1]$ for the remaining measures. We may therefore assume that $U$ is non-empty. If $V$ is also non-empty, we may finish as before by cutting at $y$ and recursively solving two smaller division problems: the division problem on $[0,y]$ for the measures in $U \cup \{t\}$, with the original demands $\alpha_i$ for each measure $\mu_i$ with $i \in U$, and demand $\alpha_t - \sum_{i \in U} \alpha_i$ for the measure $\mu_t$, and the division problem on $(y,1]$ for the measures in $V \cup \{t\}$, with the original demands $\alpha_i$ for each measure $\mu_i$ with $i \in V$, and demand $1 - \alpha_t - \sum_{i \in V} \alpha_i$ for the measure $\mu_t$.

	Thus, we are left to deal with the case where $U = Q = [n+1] \setminus \{t\}$ and $V = \emptyset$. Recall that we have $P = \emptyset$. To summarise, we now know that
	$\mu_i([0,x]) \le 1/2$ and $\mu_i([0,y]) \ge \alpha_t$ for each $i \in U$, and that $\mu_t([0,x]) = 1/2$ and $\mu_t([0,y]) = \alpha_t$. Therefore, appealing to the intermediate value theorem, we may choose a maximal $z \in (x,y]$ such that $\mu_s([0,z]) = \mu_t([0,z]) = \beta$ for some $s \in U$; by the definition of $z$, we have $\mu_i([0,z]) \ge \beta$ for each $i \in U$. Notice that we have $1 - \beta \ge 1 - \alpha_t \ge \alpha_i$ for all $i \ne t$, where the first inequality follows from monotonicity, and the second from the fact that all the demands sum to $1$. Writing $U' = U \setminus \{s\}$, we may now establish~\eqref{rec} by cutting at $z$ and recursively solving two smaller division problems: the division problem on $[0,z]$ for the measures in $U' \cup \{t\}$, with the original demands $\alpha_i$ for each measure $\mu_i$ with $i \in U'$, and demand $\beta - \sum_{i \in U'} \alpha_i$ for the measure $\mu_t$ (which is a non-negative demand since $\beta \ge 1/2$ and $\sum_{i \in U'} \alpha_i \le \sum_{i \in U} \alpha_i \le 1/2$), and the division problem on $(z,1]$ for the measures in $\{s,t\}$, with the original demand $\alpha_s$ for the measure $\mu_s$, and demand $1 - \alpha_s - \beta$ for the measure $\mu_t$.

	We now see that we may establish the estimate~\eqref{rec} in each case; this completes the proof.
\end{proof}

\section{Conclusion}\label{s:conc}
We have shown that $f(n) \le 3n-4$ for all $n \ge 2$. While we trivially have $f(n) \ge n-1$ for $n \ge 2$, it is not hard to do better. The following construction from~\citep{halevi} shows that $f(n) \ge 2n-2$ for all $n \ge 2$. Let $\mu_1$ be the uniform measure on $[0,1]$ and for $1 \le i \le n-1$, let $\mu_{i+1}$ be the uniform measure supported on the (tiny) interval $[i/n - \eps , i/n + \eps]$, where $\eps = 1/(100n)^{10}$. For the demands $\alpha_1 = 1 - \eps^{10}$ and $\alpha_i = \eps^{10} / (n-1)$ for each $2 \le i \le n$, it is easily seen that any disproportionate division requires at least two cuts in the support of each of the measures $\mu_2, \mu_3, \dots, \mu_n$, for a total of $2n-2$ cuts.

We suspect that the above construction reflects the truth and that $f(n) = 2n-2$ for all $n \ge 2$. This conjecture is particularly interesting because we believe there are deeper topological considerations driving this claim. Concretely, the following measure partitioning conjecture, if true, would imply by a simple inductive argument that $f(n) = 2n-2$ for all $n \ge 2$.

\begin{conjecture}\label{partn}
	For any $n\ge2$ non-atomic probability measures $\mu_1, \mu_2, \dots, \mu_n$ on the unit circle $S^1$ and non-negative reals $\alpha_1, \alpha_2, \dots, \alpha_n$ summing to $1$, there exists a partition of the set $[n] = P \cup Q$ into two nonempty sets and a partition of the circle $S^1 = X \cup X^c$ into two intervals such that
	\[\min_{i \in P} \mu_i(X) = \sum_{j \in P} \alpha_j \text {  and  } \min_{i \in Q} \mu_i(X^c) = \sum_{j \in Q} \alpha_j. \]
\end{conjecture}

The above conjecture for two measures is known to be true and is essentially the Stromquist--Woodall theorem~\citep{sw} for two measures, but we have unfortunately been unable to decide this conjecture even for three measures.
\section*{Acknowledgements}
The second author wishes to acknowledge support from NSF grant DMS-1800521, and the third author was supported by NSF grant DMS-1802201. We would like to thank Imre B\'ar\'any and D\"om\"ot\"or P\'alv\"olgyi for several interesting discussions around the measure partitioning conjecture presented here.

\bibliographystyle{amsplain}
\bibliography{disprop_div}

\providecommand{\bysame}{\leavevmode\hbox to3em{\hrulefill}\thinspace}
\providecommand{\MR}{\relax\ifhmode\unskip\space\fi MR }
% \MRhref is called by the amsart/book/proc definition of \MR.
\providecommand{\MRhref}[2]{%
  \href{http://www.ams.org/mathscinet-getitem?mr=#1}{#2}
}
\providecommand{\href}[2]{#2}
\begin{thebibliography}{10}

\bibitem{envy2}
H.~Aziz and S.~Mackenzie, \emph{A discrete and bounded envy-free cake cutting
  protocol for any number of agents}, 57th {A}nnual {IEEE} {S}ymposium on
  {F}oundations of {C}omputer {S}cience---{FOCS} 2016, IEEE Computer Soc., Los
  Alamitos, CA, 2016, pp.~416--427.

\bibitem{imre}
I.~B\'{a}r\'{a}ny and J.~Matou\v{s}ek, \emph{Simultaneous partitions of
  measures by {$k$}-fans}, Discrete Comput. Geom. \textbf{25} (2001), 317--334.

\bibitem{better}
S.~J. Brams, M.~A. Jones, and C.~Klamler, \emph{Better ways to cut a cake},
  Notices Amer. Math. Soc. \textbf{53} (2006), 1314--1321.

\bibitem{envy1}
S.~J. Brams and A.~D. Taylor, \emph{An envy-free cake division protocol}, Amer.
  Math. Monthly \textbf{102} (1995), 9--18.

\bibitem{dubspan}
L.~E. Dubins and E.~H. Spanier, \emph{How to cut a cake fairly}, Amer. Math.
  Monthly \textbf{68} (1961), 1--17.

\bibitem{halevi}
E.~Segal-Halevi, \emph{Cake-cutting with different entitlements: How many cuts
  are needed?}, Preprint, arXiv:1803.05470.

\bibitem{steinhaus}
H.~Steinhaus, \emph{Sur la division pragmatique}, Econometrica \textbf{17}
  (1949), 315--319.

\bibitem{stonetukey}
A.~H. Stone and J.~W. Tukey, \emph{Generalized ``sandwich'' theorems}, Duke
  Math. J. \textbf{9} (1942), 356--359.

\bibitem{strom}
W.~Stromquist, \emph{How to cut a cake fairly}, Amer. Math. Monthly \textbf{87}
  (1980), 640--644.

\bibitem{sw}
W.~Stromquist and D.~R. Woodall, \emph{Sets on which several measures agree},
  J. Math. Anal. Appl. \textbf{108} (1985), 241--248.

\bibitem{webb}
W.~A. Webb, \emph{How to cut a cake fairly using a minimal number of cuts},
  Discrete Appl. Math. \textbf{74} (1997), 183--190.

\end{thebibliography}

\end{document}